\documentclass[a4paper,12pt]{amsart}
\usepackage{amssymb}
\usepackage{amsmath}
\usepackage[utf8]{inputenc}
\usepackage{stmaryrd}
\usepackage{amscd,amsthm,amssymb}
\usepackage{enumerate}
\usepackage{color}

\scrollmode
\usepackage{latexsym}

\def\.{\cdot}

\def\d{{\delta}}

\def\la{\langle}
\def\ra{\rangle}
\def\O{\mathrm{O}}

\def\l{\lambda}

\def\beq{\begin{equation}}
\def\eeq{\end{equation}}
\def\bea{\begin{eqnarray*}}
\def\eea{\end{eqnarray*}}
\def\beaa{\begin{eqnarray}}
\def\eeaa{\end{eqnarray}}
\def\ba{\begin{array}}
\def\ea{\end{array}}

\def\L{\Lambda}

\def \RM{\mathbb{R}}

\def \CM{\mathbb{C}}

\def\i{\mathrm{i}}
\def\rk{\mathrm{rk}}

\def\id{\mathrm{id}}
\def\be{\begin{equation}}
\def\ee{\end{equation}}
\def\tr{\mathrm{tr}}

\def\rk{\mathrm{rk}}

\def\Sym{\mathrm{Sym}}

\def\SU{\mathrm{SU}}

\def\E{\mathrm{E}}
\def\G{\mathrm{G}}

\def\OO{\mathcal{O}}
\def\End{\mathrm{End}}

\def\Spin{\mathrm{Spin}}

\def\Sym{\mathrm{Sym}}

\def\Id{\mathrm{id}}

\def\T{\mathrm{\,T}}
\def\S{\mathrm{S}}

\def\lr{\,\lrcorner\,}

\def\dd{\mathrm{d}}
\def\LL{\mathrm{L}}

\newtheorem{epr}{Proposition}[section]
\newtheorem{ath}[epr]{Theorem}
\newtheorem{elem}[epr]{Lemma}
\newtheorem{ecor}[epr]{Corollary}

\theoremstyle{definition}
\newtheorem{ede}[epr]{Definition}

\newtheorem{exe}[epr]{Example}

\title[Generalized vector cross products and  Killing forms]{Generalized vector cross products and Killing forms on negatively curved manifolds}

\address{Maria Laura Barberis, FAMAF-CIEM, Universidad Nacional de C\'ordoba, Ciudad Universitaria, 5000 C\'ordoba, Argentina}
\email{barberis@famaf.unc.edu.ar}

\author{Mar\'\i a Laura Barberis, Andrei Moroianu, Uwe Semmelmann}

\address{Andrei Moroianu \\ Laboratoire de Math\'ematiques d'Orsay, Univ. Paris-Sud, CNRS, 
Universit\'e Paris-Saclay, 91405 Orsay, France }
\email{andrei.moroianu@math.cnrs.fr}

\address{Uwe Semmelmann\\
Institut f\"ur Geometrie und Topologie \\
Fachbereich Mathematik\\
Universit{\"a}t Stuttgart\\
Pfaffenwaldring 57 \\
70569 Stuttgart, Germany
}
\email{uwe.semmelmann@mathematik.uni-stuttgart.de}

\date{\today}

\begin{document}

\begin{abstract}
Motivated by the study of Killing forms on compact Riemannian manifolds of negative sectional curvature, we introduce the notion of generalized vector cross products on $\mathbb{R}^n$ and give their classification. Using previous results about Killing tensors on negatively curved manifolds and a new characterization of $\mathrm{SU}(3)$-structures in dimension $6$ whose associated $3$-form is Killing, we then show that every Killing $3$-form on a compact $n$-dimensional Riemannian manifold with negative sectional curvature vanishes if $n\ge 4$.

\end{abstract}

\maketitle

\section{Introduction}

The notion of multi-linear $p$-fold vector cross products on has been introduced by Gray and Brown \cite{gray1} who classified them over arbitrary fields of characteristic different from 2. Their relations to special types of geometries were further studied by Gray \cite{gray2}. In the present article we are mainly concerned with $2$-fold vector cross products over the real numbers and with a natural generalization of them coming from the theory of Killing tensors. 

Roughly speaking, a vector cross product on the Euclidean space $\RM^n$ is a $3$-form $\tau\in\Lambda^3\RM^n$ with the property that $X\lr\tau$ is a Hermitian 2-form on $X^\perp$ for every unit vector $X\in S^{n-1}$. Using the classification of division algebras it is easy to show that vector cross products only exist for $n=3$ and $n=7$.

In Definition \ref{gvcp} below, we introduce the notion of a {\em generalized vector cross product} which is a (non-vanishing) $3$-form $\tau\in\Lambda^3\RM^n$ with the property that $X\lr\tau$ belongs to some fixed $\O(n)$-orbit in $\Lambda^2\RM^n$ for every unit vector $X\in S^{n-1}$. 

We classify generalized vector cross products in Theorem \ref{th} below. When $n$ is odd, every generalized vector cross product is up to rescaling a vector cross product, and thus $n=3$ or $n=7$. If $n=4k+2$, generalized vector cross products can only exist for $n=6$, where they correspond to 3-forms with $\SU(3)$ stabilizer. For $n=4k$, generalized vector cross products do not exist.

In the second part of the paper we study the problem which motivated the notion above: Killing forms on $\SU(3)$-manifolds and on compact Riemannian manifolds of negative sectional curvature. 

Recall that Killing $p$-forms on Riemannian manifolds are differential forms of degree $p$ with totally skew-symmetric covariant derivative. Parallel forms are of course Killing, but there are many examples of Riemannian manifolds carrying non-parallel Killing forms. Let us quickly review the state of the art.

Killing $1$-forms are just the metric duals to Killing vector fields. There are also several interesting
examples of Killing forms of higher degree, typically related to special geometric structures. On the standard sphere the space of Killing $p$-forms coincides with the 
eigenspace of the Hodge-Laplace operator on coclosed $p$-forms for the smallest eigenvalue.
The fundamental $2$-form of a nearly K\"ahler manifold or the defining $3$-form of a nearly
parallel $\G_2$ manifold are Killing forms by definition. Killing forms also exist on Sasakian and
$3$-Sasakian manifolds. Moreover, the torsion $3$-form of a metric connection with parallel
and totally skew-symmetric torsion is Killing. We refer to \cite{uwe} for further details on Killing forms.

There are also obstructions and non-existence results about Killing forms under certain assumptions. Compactness plus some holonomy reduction often forces Killing forms to be parallel, e.g. on K\"ahler manifolds \cite{yamaguchi}, on quaternion-K\"ahler manifolds \cite{au-qk}, on locally symmetric Riemannian spaces other than spheres \cite{fau-sym}, or on manifolds with holonomy $\G_2$ or $\Spin_7$ \cite{uwe2}. The only known obstruction related to curvature holds on compact 
manifolds of constant negative sectional curvature, where a Weitzenb\"ock formula shows that every Killing forms has to vanish (cf. \cite{uwe}, Prop. 2.4). However, until now no obstruction was known in the case of non-constant negative sectional curvature.

The situation is rather different for Killing tensors (which are defined as symmetric tensors for which the complete symmetrization 
of the covariant derivative vanishes). Indeed, Dairbekov and Sharafutdinov proved in \cite{ds} that all {\em trace-free} Killing tensors on compact manifolds with negative sectional curvature vanish (cf. \cite{aku}, Prop. 6.6 for a more conceptual proof). 

Using this fact, we show in Proposition \ref{p1} below that if $(M,g)$ is a Riemannian manifold with negative sectional curvature, then the only Killing tensors  on $M$ are, up to constant rescaling, the symmetric powers of the Riemannian metric tensor $g$, and then obtain the following results concerning Killing forms on $(M,g)$:
\begin{itemize} 
\item Every non-zero Killing 2-form is parallel and defines, after constant rescaling, a Kähler structure on $(M,g)$ (Proposition \ref{52});
\item If the dimension $n$ of $M$ is different from $3$, every Killing 3-form on $M$ vanishes;  if $n=3$, every Killing 3-form on $M$ is parallel (Theorem \ref{54}).
\end{itemize}

The proof of Theorem \ref{54} relies on the above classification of generalized vector cross products (Theorem \ref{th}), and on a result of independent interest (Proposition \ref{3k}) where we show that the $3$-form corresponding to a $\SU(3)$-structure on a (not necessarily compact) 6-dimensional manifold is Killing if and only if the $\SU(3)$-structure is parallel (i.e. $M$ is non-compact Calabi-Yau), or if it is defined by a strict nearly Kähler structure on $M$.

{\sc Acknowledgment.} This work was initiated in C\'ordoba, Argentina, and completed during a ``Research in Pairs'' stay at the Mathematisches Forschungsinstitut, Oberwolfach, Germany. We thank the MFO for the excellent research conditions provided.

\section{Generalized vector cross products}

Let $\{e_i\}$ denote the standard base of the Euclidean space $(\RM^n,\la\cdot,\cdot\ra)$. We identify vectors and co-vectors using the Euclidean scalar product. 

\begin{ede} A {\em vector cross product} on $(\RM^n,\la\cdot,\cdot\ra)$  is an element $\tau\in\Lambda^3\RM^n$ such that 
\be\label{evcp}|\tau_XY|^2=|X\wedge Y|^2,\qquad\forall\ X,Y\in \RM^n\ ,\ee
where for $X\in\RM^n$, $\tau_X$ denotes the skew-symmetric endomorphism of $\RM^n$ defined by 
$$\la\tau_XY,Z\ra:=\tau(X,Y,Z),\qquad\forall\ Y,Z\in \RM^n\ .$$
\end{ede}

The above condition \eqref{evcp} is equivalent to the fact that for every unit vector $X\in \RM^n$, the skew-symmetric endomorphism $\tau_X$ restricted to $X^\perp$ is orthogonal. Therefore, if $\tau$ is a vector cross product on $\RM^n$, $\tau_X$ is a complex structure on $X^\perp$ for every unit vector $X$, so in particular $n$ has to be odd. Actually much more can be said:
\begin{epr}[cf. \cite{gray2}]\label{vcp}
Let $\tau$ be a vector cross product on $\RM^n$. Then either $n=3$, and $\tau$ is the volume form, or $n=7$ and $\tau$ belongs to the $\O(7)$-orbit of the $3$-form
\be\label{tau0}\tau_0:=e_{127}+e_{347}+e_{567}+e_{135}-e_{146}-e_{236}-e_{245}\ ,
\ee
whose stabilizer in $\O(7)$ is the exceptional group $\G_2$ (in order to simplify the notation, here and in the sequel $e_{ijk}$ stands for $e_i\wedge e_j\wedge e_k$).
\end{epr}

We now introduce the following notion, whose motivation will become clear in the next sections of this work. For any skew-symmetric endomorphism $A\in\End^-(\RM^n)$ we denote by $\OO_A\subset \End^-(\RM^n)$ the orbit of $A$  under the adjoint action of the orthogonal group $\O(n)$. 

\begin{ede} \label{gvcp} 
A {\em generalized vector cross product} on the Euclidean space $(\RM^n,\la\cdot,\cdot\ra)$ is an element $\tau\in \Lambda^3\RM^n$ such that there exists a non-vanishing skew-symmetric endomorphism $A\in \End^-(\RM^n)$ with the property that
\be\label{egvcp}\tau_X\in\OO_A,\qquad\forall\ X\in S^{n-1}\subset \RM^n\ .\ee
\end{ede}

Equivalently, $\tau$ is a generalized vector cross product if it is non-zero and if for every unit vectors $X,Y\in\S^{n-1}$, the symmetric endomorphisms $\tau_X^2$ and $\tau_Y^2$ have the same eigenvalues with the same multiplicities.

If $\tau$ is a generalized vector cross product on $\RM^n$, then $\lambda\tau$ is a generalized vector cross product on $\RM^n$ for every $\lambda\in\RM\setminus\{0\}$. The endomorphism $A$, defined up to orthogonal conjugation, is called the {\em associated endomorphism} of $\tau$.

\begin{exe}\label{exa0} Let $A_0\in\End^-(\RM^{2m+1})$ be defined by 
\be\label{a0}A_0(e_{2m+1})=0, \qquad A_0(e_{2i-1})=e_{2i},\qquad A_0(e_{2i})=-e_{2i-1},\qquad \forall \ i\in\{1,\ldots,m\}\ .\ee
Then $\tau\in\Lambda^3\RM^{2m+1}$ is a vector cross product if and only if $\tau$ is a generalized vector cross product with associated endomorphism (conjugate to) $A_0$. Indeed, $B\in \OO_{A_0}$ if and only if $B$ has a one-dimensional kernel and its restriction to $\ker(B)^\perp$ is a complex structure. As $X\in \ker(\tau_X)$ for every $X$, \eqref{evcp} is equivalent to $\tau_X\in \OO_{A_0}$ for every unit vector $X$.
\end{exe}

We now give an important example of generalized vector cross product on $\RM^6$, induced by the standard $\SU(3)$-structure:

\begin{exe}\label{su3} Consider the usual Hermitian structure $J$ on $\RM^6$ with fundamental 2-form $e_{12}+e_{34}+e_{56}\in\Lambda^2\RM^6$, and let 
\be\label{sigma}\sigma_0:=\mathrm{Re}((e_1+ie_2)\wedge(e_3+ie_4)\wedge(e_5+ie_6))=e_{135}-e_{146}-e_{236}-e_{245}\in \Lambda^3\RM^6\ee
be the real part of the complex volume form of $\Lambda^{3,0}\RM^6$. Then $(\sigma_0)_X$ is a Hermitian structure on the orthogonal complement $\mathrm{span}(X,JX)^\perp$ for every unit vector $X$. Indeed, since $\SU(3)$ acts transitively on $S^5\subset \RM^6$ and preserves $J$ and $\sigma_0$, it is enough to check this for $X=e_1$, where indeed the skew-symmetric endomorphism $A$ corresponding to $e_1\lr \sigma_0=e_{35}-e_{46}$ is a Hermitian structure on $\mathrm{span}(e_1,Je_1)^\perp$. By Definition \ref{gvcp}, $\sigma_0$ is a generalized vector cross product on the Euclidean space $\RM^6$ with associated endomorphism $A$.
\end{exe}

The next result, together with Proposition \ref{vcp}, gives the classification of generalized vector cross products in all dimensions:

\begin{ath}\label{th} $(i)$ If $n=2m+1$ is odd, every generalized vector cross product on $\RM^n$ is, up to constant rescaling, a vector cross product, and thus $n=3$ or $n=7$.

$(ii)$ If $n=4k+2$ and $\sigma$ is a generalized vector cross product on $\RM^{n}$ then $n=6$ and, up to constant rescaling, $\sigma$ belongs to the $\O(6)$-orbit of $\sigma_0$ constructed in Example \ref{su3}. 

$(iii)$ There is no generalized vector cross product in dimension $n=4k$.
\end{ath}
\begin{proof} (i) Let $n=2m+1$ and let $\tau\in \Lambda^3\RM^n$ be a generalized vector cross product with associated endomorphism $A\in \End^-(\RM^n)$. Denote by $\mu<0$ one of the eigenvalues of the symmetric endomorphism $A^2$. The corresponding eigenspace $E_\mu$ is orthogonal to $\ker(A)$. By \eqref{egvcp}, we see that for every $X\in \S^{n-1}$, the $\mu$-eigenspace $E_\mu(\tau_X^2)$ of $\tau_X^2$ is orthogonal to $\ker(\tau_X)$. Since obviously $X\in\ker(\tau_X)$, we get $E_\mu(\tau_X^2)\subset X^\perp=\T_XS^{n-1}$  for every $X\in S^{n-1}$.
Moreover, the dimension of $E_\mu(\tau_X^2)$ does not depend on $X$, so the collection $\{E_\mu(\tau_X^2)\}$ defines a sub-bundle of $\T S^{n-1}$. 

However, it is well known that even dimensional spheres do not have any proper sub-bundles of their tangent bundles. Indeed, if $\T S^{2m}=T_1\oplus T_2$ with $\rk(T_1),\rk (T_2)<2m$, then the Euler class $e(\T S^{2m})=e(T_1)\cup e(T_2)$ and the Euler classes of $T_1$ and $T_2$ vanish (as they belong to a cohomology group which is zero), contradicting the fact that the Euler class of $\T S^{2m}$ is non-zero.

Thus $E_\mu$ has to be of dimension $2m$. Denoting by $\lambda:=\sqrt{-\mu}$, we see that $\frac1\lambda A$ belongs to the orbit $\OO_{A_0}$, where $A_0$ is defined by \eqref{a0}, so $\frac1\lambda\tau$ is a vector cross product by Example \ref{exa0}. 

(ii) Suppose now that $n=4k+2$. By Definition \ref{gvcp}, $\sigma\ne 0$, so $n\ge 6$. Let $A$ be the associated endomorphism of $\sigma$ and let $\mu<0$ be one of the eigenvalues of the symmetric endomorphism $A^2$. The corresponding eigenspace $E_\mu$ has even dimension $d$ and defines a rank $d$ sub-bundle of $\T S^{4k+1}$ as before. 

On the other hand, by Thm. 27.18 in \cite{steenrod}, every sub-bundle of $\T S^{4k+1}$ has rank $0$, $1$, $4k$, or $4k+1$. Since $d$ is even and non-zero, the only possibility is $d=4k$, and therefore $A^2$ has only one non-zero eigenvalue. After rescaling $\sigma$ (and $A$), we can assume that $\mu=-1$, so for every unit vector $X$, $\sigma_X$ has exactly two eigenvalues: $0$ with multiplicity 2, and $-1$ with multiplicity $4k$. 

Let us denote by $*:\Lambda^p\RM^{n}\to\Lambda^{n-p}\RM^{n}$ the Hodge operator induced by the standard metric and orientation. For every $X\in \RM^{n}$ we view $\sigma_X$ as a 2-form in $\Lambda^2\RM^{n}$ and define
$$\psi_X:=\frac{1}{(2k)!}(\sigma_X)^{\wedge 2k},\qquad \text{where}\qquad(\sigma_X)^{\wedge 2k}:=\underbrace{\sigma_X\wedge\ldots\wedge\sigma_X}_{2k \text{ times}}\in\Lambda^{4k}\RM^n\ .$$
By the above considerations we see that $\psi_X$ is exactly the volume element of the eigenspace $E_{-1}$ of $\sigma_X^2$ for every unit vector $X$.

We now consider the map $F:\RM^{n}\to \RM^n$ defined by
$$F(X):=\begin{cases}\tfrac1{|X|^{2k}}*(X\wedge\psi_X)& \text{if } X\ne 0\\ 0& \text{if } X= 0\end{cases}\ ,$$
which is clearly continuous on $\RM^n$ and smooth on $\RM^n\setminus\{0\}$.

By construction, $F(X)$ is orthogonal to $X$ and to $E_{-1}=\mathrm{Im}(\sigma_X)=(\ker(\sigma_X))^\perp$ for every $X$, whence
\be\label{taufx}\sigma_XF(X)=0,\qquad\forall\ X\in\RM^n\ .
\ee
By construction we also have $|F(X)|=1$ if $|X|=1$ and
\be\label{ftx}F(tX)=tF(X),\qquad \forall\ t\in \RM,\ \forall X\in\RM^n\ ,\ee 
showing that $|F(X)|=|X|$ for every $X\in\RM^n$. 

The above properties of $F$ imply that $\{X,F(X)\}$ is an orthonormal basis of $\ker(\sigma_X)$ for every unit vector $X$, whence 
\be\label{t2}\sigma_X^2=-\Id+X\otimes X+F(X)\otimes F(X),\qquad\forall \ X\in S^{n-1}\ ,
\ee
where by convention, for vectors $U,V\in\RM^n$, $U\otimes V$ denotes the endomorphism of $\RM^n$ defined by $X\mapsto \la X,V\ra U.$ Using the homogeneity of $F$,  \eqref{t2} yields
\be\label{t21}\sigma_X^2=-|X|^2\Id+X\otimes X+F(X)\otimes F(X),\qquad\forall \ X\in \RM^n\ .
\ee

For every unit vector $X$ we have $\sigma_{F(X)}X=-\sigma_X F(X)=0$, showing that $X\in\ker(\sigma_{F(X)})=\mathrm{span}(F(X),F(F(X)))$, and as $X\perp F(X)$, we get \be\label{ff}F(F(X))=\pm X\ .\ee 
By \eqref{ftx}, this holds for every $X\in \RM^n$. 

Our aim is to show that $F$ is linear. In order to do this, we polarize \eqref{t21} and obtain for every $X,Y\in\RM^n$:
\bea\sigma_X\circ\sigma_Y+\sigma_Y\circ\sigma_X&=&F(X+Y)\otimes F(X+Y)-F(X)\otimes F(X)-F(Y)\otimes F(Y)\\&&-2\la X,Y\ra\Id+X\otimes Y+Y\otimes X\ .
\eea
This shows that the map $H:\RM^n\times\RM^n\to\RM^n\otimes \RM^n\simeq\End(\RM^n)$ defined by
\be\label{h}H(X,Y):=F(X+Y)\otimes F(X+Y)-F(X)\otimes F(X)-F(Y)\otimes F(Y)\ee
is bilinear, so in particular $H(tX,Y)=tH(X,Y)$ for every $t\in\RM$ and $X,Y\in\RM^n$. 

We now fix $X,Y\in\RM^n\setminus\{0\}$ and compute:
\bea H(X,Y)&=&\lim_{t\to 0}\frac1t{H(tX,Y)}\\&=&\lim_{t\to 0}\frac1t\left({F(tX+Y)\otimes F(tX+Y)-t^2F(X)\otimes F(X)-F(Y)\otimes F(Y)}\right)\\
&=&\lim_{t\to 0}\frac1t\left({[F(tX+Y)-F(Y)]\otimes F(tX+Y)+F(Y)\otimes[F(tX+Y)- F(Y)]}\right)\\
&=&\dd F_Y(X)\otimes F(Y)+F(Y)\otimes \dd F_Y(X)\ ,
\eea
where the differential $\dd F_Y$ of $F$ at $Y$ is well-defined as $Y\ne 0$. In particular we get:
\be\label{zz}\la H(X,Y),Z\otimes Z\ra=0,\qquad \forall \ Z\in F(Y)^\perp\ ,
\ee
where the scalar product here is the canonical extension of the Euclidean scalar product to $\RM^n\otimes \RM^n$ defined by $\la U\otimes V,A\otimes B\ra:=\la U,A\ra\la V,B\ra$.

Using \eqref{h} we get for every $Z$ orthogonal to $F(X)$ and $F(Y)$:
$$0=\la H(X,Y),Z\otimes Z\ra=\la F(X+Y),Z\ra^2\ ,$$
whence 
\be\label{fxy}F(X+Y)\in\mathrm{span}(F(X),F(Y)),\qquad\forall\ X,Y\in\RM^n\ee
(this is tautologically true when $X=0$ or $Y=0$). 

We now define the open set 
$$\mathcal{U}:=\{(X,Y)\in\RM^n\times\RM^n\ |\ X\wedge Y\ne 0\}\ .$$

The set $\mathcal{U}$ is obviously dense and path connected (recall that $n\ge 6$). By \eqref{ff}, for every $(X,Y)\in \mathcal{U}$, $F(X)$ is not collinear to $F(Y)$, so by \eqref{fxy} there exist uniquely defined maps $a,b:\mathcal{U}\to\RM$ such that 
\be\label{fxy1}F(X+Y)=a(X,Y)F(X)+b(X,Y)F(Y),\qquad\forall\ (X,Y)\in \mathcal{U}\ .\ee

For every $(X,Y)\in \mathcal{U}$ there exists $Z\in F(Y)^\perp$ such that $\la F(X),Z\ra\ne 0$. Using \eqref{h} and \eqref{zz} we obtain:
$$0=\la H(X,Y),Z\otimes Z\ra=\la F(X+Y),Z\ra^2-\la F(X),Z\ra^2=(a(X,Y)^2-1)\la F(X),Z\ra^2\ ,$$
showing that the image of $a$ is contained in $\{\pm1\}$. By symmetry, the same holds for $b$. 

We now notice that $a$ and $b$ are continuous functions on $\mathcal{U}$. Indeed, taking the scalar product in \eqref{fxy1} with the vector $Z(X,Y):=F(X)-\frac{\la F(X),F(Y)\ra F(Y)}{|F(Y)|^2}$ (which is orthogonal to $F(Y)$ and continuous on $\mathcal{U}$) yields
$$\la F(X+Y),Z(X,Y)\ra=a(X,Y)\frac{|F(X)|^2|F(Y)|^2-\la F(X),F(Y)\ra^2}{|F(Y)|^2}\ ,$$
whence 
$$a(X,Y)=\frac{\la F(X+Y),Z(X,Y)\ra|F(Y)|^2}{|F(X)|^2|F(Y)|^2-\la F(X),F(Y)\ra^2}$$
is well-defined and continuous by the Cauchy-Schwarz inequality. A similar argument shows that $b$ is continuous too, and since $\mathcal{U}$ is connected, $a$ and $b$ are constant on $\mathcal{U}$. 

Remark that if $(X,Y)\in \mathcal{U}$, then $(X+Y,-Y)\in \mathcal{U}$. Applying \eqref{fxy1} twice and using \eqref{ftx} we get
$$F(X)=F((X+Y)+(-Y))=aF(X+Y)+bF(-Y)=a^2F(X)+abF(Y)-bF(Y)\ .$$
Since $a^2=1$ and $F(Y)\ne 0$, we get $ab=b$, whence $a=1$ and similarly $b=1$. We thus have $F(X+Y)=F(X)+F(Y)$ for all $(X,Y)\in \mathcal{U}$, so using the density of $\mathcal{U}$ and the continuity of $F$, together with \eqref{ftx} we obtain that $F$ is linear. Moreover, it was noticed above that $F$ is norm-preserving, and $F(X)\perp X$ for every $X$, so $F$ is a skew-symmetric orthogonal map, i.e. a Hermitian structure on $\RM^n$.

We consider now the standard embedding $\RM^n\subset\RM^{n+1}\simeq\RM^n\oplus \RM e_{n+1}$ and define $\tau\in\Lambda^3\RM^{n+1}$ by
$$\tau:=e_{n+1}\wedge F+\sigma$$
($F$ is viewed here as an element of $\Lambda^2\RM^n\subset\Lambda^2\RM^{n+1}$). We claim that $\tau$ is a vector cross product on $\RM^{n+1}$. To see this, let $\tilde X:=X+ae_{n+1}$ and $\tilde Y:=Y+be_{n+1}$ be two arbitrary vectors in $\RM^{n+1}$ (with $a,b\in \RM$ and $X,Y\in \RM^n$).
We compute 
\be\label{ti} |\tilde X\wedge\tilde Y|^2=|X\wedge Y+ae_{n+1}\wedge Y-be_{n+1}\wedge X|^2=|X\wedge Y|^2+|aY-bX|^2
\ee
and
$$\tau_{\tilde X}\tilde Y=(aF-e_{n+1}\wedge F(X)+\sigma_X)(\tilde Y)=aF(Y)-bF(X)+e_{n+1}F(X,Y)+\sigma_XY\ .$$
Using \eqref{t21} we get
\bea|\sigma_XY|^2&=&-\la\sigma_X^2Y,Y\ra=-\la-|X|^2Y+X\la X,Y\ra+F(X)F(X,Y),Y\ra\\&=&|X|^2|Y|^2-\la X,Y\ra^2-F(X,Y)^2\ .\eea
Moreover, $\sigma_XY$ is orthogonal to both $F(X)$ and $F(Y)$ by \eqref{taufx}. Using the last two relations together with the fact that $F$ is orthogonal we thus compute
\bea |\tau_{\tilde X}\tilde Y|^2&=&|aF(Y)-bF(X)|^2+F(X,Y)^2+|\sigma_XY|^2\\
&=&|aF(Y)-bF(X)|^2+|X|^2|Y|^2-\la X,Y\ra^2\\
&=&|aY-bX|^2+|X\wedge Y|^2\ ,
\eea
which, together with \eqref{ti}, shows that $\tau$ satisfies the definition of vector cross products \eqref{evcp}.

By Proposition \ref{vcp} we obtain that $n+1=7$, and there exists $\alpha\in \O(7)$ such that $\tau=\alpha\tau_0$, where $\tau_0$ was defined in \eqref{tau0}. Then 
$$\sigma=\tau-e_{7}\wedge(e_{7}\lr\tau)=\alpha(\tau_0-\xi\wedge(\xi\lr\tau_0)),\qquad\text{where}\qquad\xi:=\alpha^{-1}(e_7)\ .$$ 

As the group $\G_2$ acts transitively on $S^6$ and fixes $\tau_0$, there exists $\beta\in\G_2\subset \O(7)$ such that $\beta(e_7)=\xi$, whence
$$\sigma=\alpha(\tau_0-\xi\wedge(\xi\lr\tau_0))=\alpha\beta(\tau_0-e_7\wedge(e_7\lr\tau_0))\ .$$
Moreover $\alpha\beta(e_7)=e_7$ so $\alpha\beta\in\O(6)$ and the above relation shows that $\sigma$ is in the $\O(6)$-orbit of the $3$-form $\tau_0-e_7\wedge(e_7\lr\tau_0)$ which by \eqref{tau0} is equal to $\sigma_0$ defined in \eqref{sigma}. This finishes the proof of part (ii) in the theorem.

(iii) Let $n=4k$ and assume that $\sigma\in\Lambda^3\RM^{n}$ is a generalized vector cross product with associated endomorphism $A$. By definition, $\sigma_X$ is conjugate to $A$ for every unit vector $X$, and since $\sigma_X(X)=0$, $A$ is non-invertible. Moreover, $A$ is skew-symmetric, so its kernel has even dimension $2d$. 

\begin{elem} \label{le1}
For every non-zero $U\in\ker(\sigma_X)$ one has $\ker(\sigma_U)=\ker(\sigma_X)$. 
\end{elem}

\begin{proof} 
For every non-zero vector $X$, $\ker(\sigma_X)$ has dimension $2d$, so the $(4k-2d)$--form $(\sigma_X)^{\wedge (2k-d)}$ is non-zero, and the
$(4k-2d+2)$--form $(\sigma_X)^{\wedge (2k-d+1)}$ vanishes (here we view $\sigma_X$ as a 2-form rather than as a skew-symetric endomorphism). Replacing $X$ by $X+tW$ and considering the coefficient of $t$ in the formula 
\beq\label{sx}0=(\sigma_X+t\sigma_W)^{\wedge (2k-d+1)}\eeq
we obtain 
$$0=(\sigma_X)^{\wedge (2k-d)}\wedge\sigma_W,\qquad\forall\ X,W\in\RM^{4k}\ .$$
Let $X$ be a unit vector, and $U,V\in\ker(\sigma_X)$. Taking the interior product with $U$ and then with $V$ in the above formula yields
$$0=(\sigma_X)^{\wedge (2k-d)}\sigma_W(U,V),\qquad\forall \ W\in\RM^{4k}\ ,$$
whence $\sigma_U(V)=0$ for all $U,V\in\ker(\sigma_X)$. This shows that $\ker(\sigma_X)\subset \ker(\sigma_U)$ for all $U\in\ker(\sigma_X)$, and since $\dim(\ker(\sigma_X))=\dim(\ker(\sigma_U))$ for every $U$ non-zero, the lemma is proved.

\end{proof}

\begin{elem} \label{le2}
The dimension $2d$ of $\ker(\sigma_X)$ is a multiple of $4$ for every unit vector $X$.
\end{elem}
\begin{proof} For every unit vector $X$, the wedge product $\omega_X:=(\sigma_X)^{\wedge (2k-d)}$ is a volume form of $(\ker(\sigma_X))^\perp$, whose norm is independent of $X$. Let $X_t$ be a continuous path in $\ker(\sigma_X)$ with $X_0=X$,  $X_1=-X$, and $|X_t|=1$ for every $t\in[0,1]$. By Lemma \ref{le1}, $\omega_{X_t}$ is a volume form of constant length on $(\ker(\sigma_X))^\perp$, so by continuity it is constant. We thus have $\omega_X=\omega_{-X}$, whence $d$ is even.

\end{proof}

Let $X$ be a non-zero vector in $\RM^n$ and consider the $(n-2d)$-dimensional Euclidean space $E_X:=(\ker(\sigma_X))^\perp$ with the scalar product induced from $\RM^n$. By Lemma \ref{le1}, $\sigma_Y$ defines an automorphism of $E_X$ for every $Y\in \ker(\sigma_X)$.

\begin{elem} \label{le3}
For every non-zero vector $X\in \RM^n$ and for every $Y,Z\in\ker(\sigma_X)$, the following relation holds in $\End(E_X)$
\beq\label{com}\sigma_Y\circ(\sigma_X)^{-1}\circ\sigma_Z=\sigma_Z\circ(\sigma_X)^{-1}\circ\sigma_Y\ ,\eeq
where $(\sigma_X)^{-1}$ denotes the inverse of the automorphism $\sigma_X$ of $E_X$.
\end{elem}

\begin{proof} The annulation of the coefficient of $t^2$ in \eqref{sx} yields 
$$0=(\sigma_X)^{\wedge (2k-d-1)}\wedge\sigma_W\wedge\sigma_W,\qquad\forall\ X,W\in\RM^{4k}\ .$$
Taking the interior product in this relation with two vectors $Y$ and $Z$ from $\ker(\sigma_X)$, and using the fact that by Lemma \ref{le1} $\sigma_Y(Z)=0$, we obtain
\beq\label{ex}0=(\sigma_X)^{\wedge (2k-d-1)}\wedge\sigma_W(Y)\wedge\sigma_W(Z)=(\sigma_X)^{\wedge (2k-d-1)}\wedge\sigma_Y(W)\wedge\sigma_Z(W),\qquad\forall\ W\in\RM^{4k}\ .\eeq
The exterior form $(\sigma_X)^{\wedge (2k-d)}$ defines an orientation of $E_X$. With respect to this orientation, the Hodge dual of $(\sigma_X)^{\wedge (2k-d-1)}$ is proportional to the $2$-form of $E_X$ corresponding to the skew-symmetric endomorphism $(\sigma_X)^{-1}$. The above equation \eqref{ex} thus implies
$$0=\la \sigma_Y(W)\wedge\sigma_Z(W),(\sigma_X)^{-1}\ra=\la \sigma_Z(W),(\sigma_X)^{-1}(\sigma_Y(W))\ra=-\la W,\sigma_Z\circ(\sigma_X)^{-1}\circ\sigma_Y(W)\ra$$
for every $W\in\E_X$. This shows that the endomorphism $\sigma_Z\circ(\sigma_X)^{-1}\circ\sigma_Y$ of $E_X$ is skew-symmetric, which is equivalent to the statement of the lemma.

\end{proof}

\begin{ecor} The relation 
\beq\label{com1}\sigma_Y\circ(\sigma_U)^{-1}\circ\sigma_Z=\sigma_Z\circ(\sigma_U)^{-1}\circ\sigma_Y\eeq
holds for every $Y,Z,U\in \ker(\sigma_X)$. 
\end{ecor}
\begin{proof} Since by Lemma \ref{le1} one has $E_X=E_U$ for every non-zero vector $U\in\ker(\sigma_X)$, one can replace $X$ by $U$ in \eqref{com}.

\end{proof}

Let $Y,Z,U,V$ be non-zero vectors in $\ker(\sigma_X)$. By \eqref{com1} we have 
$$\sigma_Y\circ(\sigma_U)^{-1}\circ\sigma_Z=\sigma_Z\circ(\sigma_U)^{-1}\circ\sigma_Y\qquad\text{and}\qquad\sigma_V\circ(\sigma_U)^{-1}\circ\sigma_Z=\sigma_Z\circ(\sigma_U)^{-1}\circ\sigma_V\ .$$
Composing on the right with the inverse of $\sigma_Y$ in the first equation and with the inverse of $\sigma_V$ in the second one we obtain
$$(\sigma_Y)^{-1}\circ\sigma_Z\circ(\sigma_U)^{-1}\circ\sigma_Y=(\sigma_V)^{-1}\circ\sigma_Z\circ(\sigma_U)^{-1}\circ\sigma_V\ ,$$
and finally, composing on the left with $\sigma_V$ and on the right with $(\sigma_Y)^{-1}$ yields 
$$\sigma_V\circ(\sigma_Y)^{-1}\circ\sigma_Z\circ(\sigma_U)^{-1}=\sigma_Z\circ(\sigma_U)^{-1}\circ\sigma_V\circ(\sigma_Y)^{-1}\ .$$
This shows that for every non-zero vectors $Y,Z,U,V\in\ker(\sigma_X)$, the automorphisms $f_{VY}:=\sigma_V\circ(\sigma_Y)^{-1}$ and $f_{ZU}:=\sigma_Z\circ(\sigma_U)^{-1}$ of $E_X$ commute. Consequently, they have a common eigenvector, say $\xi$, in $E_X^\CM$. Denoting by $F\subset E_X$ the subspace generated by the real and imaginary parts of $\xi$ (of real dimension 1 or 2), we deduce that $f_{YZ}(F)=F$ for every non-zero vectors $Y,Z\in\ker(\sigma_X)$, which by the definition of $f_{YZ}$ is equivalent to $\sigma_Z^{-1}(F)=\sigma_Y^{-1}(F)$. This just means that the subspace $G:=(\sigma_Y)^{-1}(F)$ is independent on the non-zero vector $Y\in\ker(\sigma_X)$, i.e. $\sigma_Y$ is an isomorphism from $G$ to $F$ for every non-zero vector $Y\in\ker(\sigma_X)$. 

Let us fix $W\in G\setminus\{0\}$. By Lemma \ref{le2}, the dimension of $\ker(\sigma_X)$ is $2d\ge 4$, so the linear map 
$$\ker(\sigma_X)\to F,\qquad Y\mapsto \sigma_Y(W)$$
is non-injective. It follows that there exists a non-zero vector $Y\in\ker(\sigma_X)$ such that $\sigma_Y(W)=0$. By Lemma \ref{le1}, $W\in \ker(\sigma_Y)=\ker(\sigma_X)$. This contradicts the fact that $W$ is a non-zero element in $G\subset E_X=(\ker(\sigma_X))^\perp$, and concludes the proof of the theorem.

\end{proof}

\section{$\SU(3)$-structures and Killing forms}

\subsection{$\SU(3)$-structures}

A $\SU(3)$-structure on a $6$-dimensional Riemannian manifold $(M,g)$ is defined as a structure group
reduction from $\O(6)$ to $\SU(3)$. The reduction defines an almost complex structure $J$ compatible with the 
Riemannian metric $g$, and a $3$-form $\Psi^+$ of type $(3,0)+(0,3)$ such that 
$\frac14\Psi^+ \wedge \ast \Psi^+$ is equal to the Riemannian volume form. In particular, $\Psi^+$ and 
$\Psi^- := \ast \Psi^+$ both have constant length and span the space of forms of type $(3,0) + (0,3)$.

The space
of $J$-anti-invariant $2$-forms, i.e. forms of type $(2,0)+(0,2)$, can be identified with the tangent space $\T M$.
Indeed, any such form can be written as $X \lr \Psi^-$ for some tangent vector $X$.

Differentiating the relation $J^2= - \id$ with respect to the Levi-Civita covariant derivative $\nabla$ of $g$  yields $\nabla_X J \circ J + J \circ \nabla_XJ =0$, so  $\nabla_X J$ is a form of type $(2,0)+(0,2)$ for every $X$. Thus there exists a endomorphism $\alpha$ of the tangent bundle satisfying
\be\label{alpha}\nabla_XJ = \alpha(X) \lr \Psi^-,\qquad\forall \ X\in\T M\ . \ee

Every skew-symmetric endomorphism $A \in \End^- (\T M)$  acts on the exterior bundle by
\be\label{act}
A \cdot \eta := \sum_i e_i \wedge e_i \lr \eta,\qquad\forall\ \eta\in\Lambda^*M\ ,
\ee
where $\{e_i\}$ is some local orthonormal basis. The almost complex structure $J$ acts trivially on $(p,p)$-forms. Moreover, as $\Psi^-_X = - \Psi^+_{JX}$ for every tangent vector $X$, we get from \eqref{act} that $J\cdot \Psi^\pm = \pm 3 \Psi^\mp$. 

The action of the $2$-form $\Psi^-_X$ on the 
$3$-forms $\Psi^\pm$ is given by
\be\label{action}
\Psi^-_X \cdot \Psi^+ \, =\,  - 2 X \wedge \omega
\qquad\mbox{and}\qquad
\Psi^-_X \cdot \Psi^- \,=\,  - 2 JX \wedge \omega
\ee
for all tangent vectors $X$, where $\omega$ denotes the fundamental $2$-form defined by
$\omega(X, Y) := g(JX, Y)$ (cf. \cite{au1}, Lemma 2.1).

\begin{ede}[cf. \cite{gray3}]\label{dnk} A {\em nearly K\"ahler} structure on $(M,g)$ is a Hermitian structure $J$ satisfying $(\nabla_XJ)(X)=0$ for every tangent vector $X$. The structure is called {\em strict} if $\nabla_XJ\ne 0$ for every $X\ne 0$. 
\end{ede}

\begin{epr}[cf. \cite{gray3}, \cite{uwe}]\label{pnk}
Every strict $6$-dimensional nearly Kähler manifold $(M,g,J)$ is Einstein with positive scalar curvature equal to $30\lambda^2$ for some $\lambda\in\RM\setminus\{0\}$. The fundamental $2$-form $\omega$ together with the $3$-forms $\Psi^+:=\frac1{3\lambda}\dd\omega$ and $\Psi^-:=*\Psi^+$ define
 a $\SU(3)$ structure on $M$ which moreover satisfies
 \beaa\label{4}\nabla_X\omega\ &=&\ \lambda X\lr\Psi^+\\
 \label{5}\nabla_X\Psi^-&=&-\tfrac12\lambda X\lr(\omega\wedge\omega)\ .
 \eeaa
\end{epr}

More details on $\SU(3)$-structures and 6-dimensional nearly Kähler manifolds can be found in \cite{au1}. 

\subsection{Killing forms}

Let $(M, g)$ be a Riemannian manifold with Levi-Civita connection $\nabla$. 
A  {\it Killing $p$-form} on $M$ is a $p$-form $\omega$ such that 
$\nabla \omega$ is totally skew-symmetric, or equivalently, satisfying (cf. \cite{uwe}):
\be\label{kf}
\nabla_X\omega = \tfrac{1}{p+1} X \lr \dd \omega,\qquad\forall \ X\in \T M\ .
\ee

\begin{exe}\label{rkf}From Definition \ref{dnk} we see that the fundamental 2-form $\omega$ of an almost Hermitian structure $(g,J)$ is a Killing 2-form if and only if $(g,J)$ is nearly Kähler.
\end{exe}

\begin{exe}\label{rkf1} If $(g,J)$ is nearly Kähler structure on $M$ and $(\omega,\Psi^+,\Psi^-)$ denotes the associated $\SU(3)$-structure,  then \eqref{5} shows that $\Psi^-$ is a Killing 3-form. 
\end{exe}

Conversely, we have the following:

\begin{epr}\label{3k}
Let $(M, g, \omega,\Psi^\pm)$ be a Riemannian manifold with a $\SU(3)$-structure such that $\Psi^-$ is
a Killing $3$-form. Then either $(g,\omega)$ is Kähler and $\Psi^\pm$ are parallel, or $(g,\omega)$ is strict nearly K\"ahler and $\Psi^-$ is a constant multiple of $*\dd\,\omega$ as in Proposition \ref{pnk}.
\end{epr}
\proof Let $J$ denote the Hermitian structure determined by $(g,\omega)$.
Since  $\Psi^-$ is assumed to be a Killing form we have 
\be\label{p-}\nabla_X \Psi^- = \tfrac14 X \lr \dd \Psi^-\ ,\qquad\forall\ X\in \T M\ee
and by Hodge duality 
\be\label{p+}\nabla_X \Psi^+ = - \tfrac14 X \wedge \delta \Psi^+ \ ,\qquad\forall\ X\in \T M\ .
\ee

We claim that $J$ acts trivially on $\delta \Psi^+$ and $\dd \Psi^-$:
\be\label{j} J\cdot\delta \Psi^+=0,\qquad J\cdot \dd\Psi^-=0\ .\ee
Indeed the norm of 
$\Psi^+$ is constant and thus taking a covariant derivative and using \eqref{p+} gives for every tangent vector $X$:
$$
0 = g(\nabla_X \Psi^+, \Psi^+) = - \tfrac14 \,  g(X \wedge \delta \Psi^+, \Psi^+)
= - \tfrac14 \, g(\delta \Psi^+, \Psi^+_X) \ .
$$
It follows that $\delta \Psi^+$ is orthogonal to the space of $J$ anti-invariant $2$-forms, so
is a $(1,1)$-form, i.e. $J\cdot\delta \Psi^+=0$. By Hodge duality, $\dd\Psi^-$ is a $(2,2)$-form, so $J\cdot \dd\Psi^-=0$, thus proving our claim.

Taking the covariant derivative of the equation $ 3\,  \Psi^- = J \cdot \Psi^+$, and using \eqref{alpha} and \eqref{action}, together with \eqref{p-}--\eqref{j} yields:
\bea
\tfrac34 \, X \lr \dd \Psi^- &=& \nabla_XJ \cdot \Psi^+ + J \cdot \nabla_X\Psi^+
\;=\;
\Psi^-_{\alpha(X)} \cdot \Psi^+  \,-\, \tfrac14 JX \wedge  \delta \Psi^+ \\[1ex]
&=&\
-\, 2 \, \alpha(X) \wedge \omega \;-\; \tfrac14 \, JX \wedge  \delta \Psi^+ \ ,  
\eea
whence by replacing $X$ with $JX$:
\be\label{two} \tfrac34 \,JX \lr \dd \Psi^- =-\, 2 \, \alpha(JX) \wedge \omega \;+\; \tfrac14 \, X \wedge  \delta \Psi^+\ ,\qquad\forall\ X\in \T M\ .
\ee

Similarly, taking the covariant derivative of the equation $- 3 \Psi^+ = J \cdot \Psi^-$ and using \eqref{alpha} and  \eqref{action}, together with \eqref{p-}--\eqref{j}, we obtain
\be\label{eu}
\tfrac34 \, X \wedge \delta \Psi^+ =
\Psi^-_{\alpha(X)} \cdot \Psi^-  \,+\, \tfrac14 JX \lr \dd\Psi^-
\;=\;
-\, 2\, J\alpha(X) \wedge \omega \,+\, \tfrac14 JX \lr \dd\Psi^- \ .
\ee
Taking a wedge product with $X$ in  \eqref{eu} and multiplying by $3$, we obtain:
\bea
0 &=& -\, 6 \, X \wedge J\alpha(X) \wedge \omega \;+\; \tfrac34 \, X \wedge JX \lr \dd \Psi^- \ ,
\eea
which together with \eqref{two} yields
\bea
0 &=&X \wedge ( -\, 6 \, J\alpha(X)  \;-\; 2\,  \alpha(JX) ) \wedge \omega \ .
\eea

Since the wedge product with  $\omega$ is injective on 2-forms, we 
obtain $$X \wedge ( -\, 6 \, J\alpha(X)  - 2\,  \alpha(JX) ) = 0\ ,\qquad\forall \ X\in\T M\ , $$ 
and thus there exists some function $f$ on $M$ such that 
\be\label{d1}3 \, J \alpha(X) + \alpha(JX) = f X\ ,\qquad\forall\ X\in\T M\ .\ee
Applying $J$ to this equation and replacing $X$ with $JX$ yields 
\be\label{d2}- 3 \alpha(JX) - J\alpha(X) = - f X\ ,\qquad\forall\ X\in\T M\ .\ee
Multiplying \eqref{d1} by 3 and adding to \eqref{d2} finally implies
$\alpha(X) = - \tfrac{f}{4} JX$.
By \eqref{alpha} we thus get 
$$\nabla_XJ = -\tfrac{f}{4}\, \Psi^-_{JX} = -\tfrac{f}{4}\,  \Psi^+_X\ ,\qquad\forall\ X\in\T M\ . $$ 
This shows that $(\nabla_XJ)(X)=0$ for every $X$, so $(M,g,J)$ is nearly K\"ahler. Moreover, the previous equation yields
\be\label{oo}*\dd\omega=-\tfrac{3f}4*\Psi^+=-\tfrac{3f}4\Psi^-\ .
\ee

By Thm. 5.2 in \cite{gray3}, $f$ is constant on $M$. If this constant is non-zero, then $(M,g,J)$ is strict nearly Kähler and $\Psi^-$ is a constant multiple of $*\dd\omega$ by \eqref{oo}. If $f=0$, the endomorphism $\alpha$ vanishes, so $(M,g,J)$ is Kähler by \eqref{alpha}. Moreover, when $\alpha=0$, comparing \eqref{two} and \eqref{eu} immediately gives $\dd\Psi^-=0$ and $\d\Psi^+=0$, so $\Psi^\pm$ are parallel by \eqref{p-}--\eqref{p+}.

\qed

\section{Killing tensors and Killing forms on negatively curved manifolds}

\subsection{Killing tensors}
We will use the formalism introduced in \cite{aku}. For the convenience of the
reader, we recall here the standard definitions and formulas which are relevant in the sequel.

Let $(\T M, g)$ be the tangent bundle of a $n$-dimensional Riemannian manifold $(M,g)$. We denote with
$\Sym^p \T M \subset \T M^{\otimes p}$ the $p$-fold symmetric tensor product of $\T M$. 

Let $\{e_i\}$ denote from now on a local orthonormal frame of $(\T M, g)$. Using the metric $g$, we will identify $\T M$ with $\T^* M$ and thus $\Sym^2 \T^* M\simeq\Sym^2\T M$. Under this identification we view the metric tensor as a symmetric 2-tensor
$\LL := 2 g = \sum_i e_i \cdot e_i$.
The metric adjoint of $\LL\cdot:\Sym^{p} \T M \rightarrow \Sym^{p+2} \T M$ is the bundle homomorphism
$$
\L : \Sym^{p+2}\T M \rightarrow \Sym^{p} \T M , \quad K \mapsto  \sum_i e_i \lr e_i \lr  K\ .
$$

We denote by
$
\Sym^p_0 \T M := \ker( \L|_{\Sym^p \T M})
$
the space of trace-free symmetric $p$-tensors.
On sections of $\Sym^p \T M$ we define a first order differential operator
$$
\dd : \Gamma(\Sym^p \T M) \rightarrow \Gamma( \Sym^{p+1}\T M), \quad K \mapsto \sum_i e_i \cdot \nabla_{e_i}K \ ,
$$
which acts as derivation on symmetric products and commutes with $\LL\cdot$.

A symmetric tensor $K \in \Gamma(\Sym^p \T M)$ is called {\em conformal Killing tensor} if
there exists some symmetric tensor $k \in \Gamma( \Sym^{p-1}\T M)$ with
$\;
\dd  K = \LL  \cdot k
$.
The tensor $K$ is called {\em Killing} if $\dd K = 0$ or equivalently if $(\nabla_X K)(X,\ldots, X)= 0$ holds for all vector fields $X$. A Killing tensor is in particular a
conformal Killing tensor.

In \cite{aku}, Prop. 6.6 we have shown that if $(M^n,g)$ is a compact Riemannian manifold with negative sectional curvature, then every trace-free conformal Killing tensor vanishes (cf. \cite{ds} for the original proof). As a corollary we immediately obtain:

\begin{epr} \label{p1}
 If $(M^n,g)$ is a compact Riemannian manifold with negative sectional curvature, then every Killing tensor is proportional to some power of the metric (and in particular has even degree).
\end{epr}

\begin{proof} Let $K$ be a Killing tensor of degree $p$. We use induction on $p$. If $p=0$ the tensor $K$ is a function satisfying $\dd K=0$, so it is constant. If $p=1$, $K$ is a Killing vector field, so it has to vanish from the Bochner formula. Assume that $K$ is Killing and that $p\ge 2$. 

Using the standard decomposition of symmetric tensors introduced in \cite{aku}, p. 385, we can write $K=K_0+\LL \cdot K_1$, where  $K_0$ is a symmetric trace-free tensor of degree $p$ and $K_1$ is a symmetric tensor of degree $p-2$. Since $K$ is Killing, we get $0=\dd K=\dd K_0+\LL\cdot \dd K_1$, so $K_0$ is a trace-free conformal Killing tensor, which vanishes by \cite[Prop. 6.6]{aku}. Thus $K=\LL\cdot K_1$ and $0=\LL\cdot \dd K_1$, whence $K_1$ is Killing (as the operator $\LL\cdot$ is injective). By the induction hypothesis we thus have $p-2=2k$ is even and $K^1=\lambda g^k$ for some constant $\lambda$. This shows that $K=L\cdot\lambda g^k= 2\lambda g^{k+1}.$

\end{proof}

\subsection{Killing forms on negatively curved manifolds}
As a first application of the above result we have:

\begin{epr}\label{52} If $(M^n,g)$ is a compact Riemannian manifold with negative sectional curvature, then every non-zero Killing $2$-form is parallel, and defines after constant rescaling a Kähler structure on $M$.
\end{epr}

\begin{proof}
Assume that $\omega$ is a non-zero Killing 2-form on $M$. Then the symmetric 2-tensor $T$ defined by $T(X,X):=|X\lr\omega|^2$ is a Killing tensor. Indeed, one has 
$$(\nabla_XT)(X,X)=2g(X\lr\nabla_X\omega,X\lr\omega)=0,\qquad\forall\  X\in\T M\ .$$

Proposition \ref{p1} above then shows that there exists some (positive) constant $\l$ such that $|X\lr\omega|^2=\l |X|^2$ for every tangent vector $X$, thus by rescaling, one may assume that $\omega$ is the fundamental 2-form of a Hermitian structure $J$ on $(M,g)$. The fact that $\omega$ is a Killing 2-form just means that $\nabla\omega$ is totally skew-symmetric, i.e. $(M,g,J,\omega)$ is nearly Kähler. 

The universal covering of a compact nearly Kähler manifold $(M,g,\omega)$ is isometric to a product between a Kähler manifold $(M_1,g_1,\omega_1)$ and a strict nearly Kähler manifold $ (M_2,g_2,\omega_2)$ (cf. \cite{kiri} or \cite{nagy1}, Prop. 2.1). Moreover, if $M_2$ is not reduced to a point, its Ricci tensor is positive definite by Thm. 1.1 in \cite{nagy1}. The assumption on $M$ thus shows that $M=M_1$ and $\omega=\omega_1$ is a Kähler form, therefore it is parallel.

\end{proof}

We now turn our attention to Killing 3-forms and start with the following:

\begin{elem}\label{lt} If $\tau$ is a Killing $3$-form on a compact Riemannian manifold $(M^n,g)$ with negative sectional curvature, then $\tau$ has constant norm. If $\tau$ is not identically zero, then for every $p\in M$, $\tau_p$ defines a generalized vector cross product on $(T_p M, g_p)$.
\end{elem}

\begin{proof}
For every vector field $X$ we denote by $\tau_X:=X\lr \tau$ the 2-form on $M$, identified also with a skew-symmetric endomorphism of $\T M$. For every $k\ge 1$ we define the symmetric tensor $T_k$ of degree $2k$ by $T_k(X,\ldots,X)=\tr(\tau_X^{2k})$. We readily compute 
$$(\nabla_XT_k)(X,\ldots,X)=2k\,\tr(\tau_X^{2k-1}\circ(\nabla_X\tau)_X)=0\ ,$$ 
since $(\nabla_X\tau)_X=X\lr (\nabla_X\tau)=0$, as $\tau$ is Killing. 

By Proposition \ref{p1} there exist real constants $a_k$ such that 
\be\label{trace}\tr(\tau_X^{2k})=a_k|X|^{2k},\qquad\forall\ X\in\T M,\ \forall\ k\ge 1\ .\ee

In particular, \eqref{trace} for $k=1$ shows that $\tau$ has constant norm on $M$. Indeed, using a local orthonormal basis $\{e_i\}$ of $\T M$ we compute:
\be\label{cn}|\tau|^2=\sum_ig(\tfrac13e_i\wedge e_i\lr\tau,\tau)=\tfrac13\sum_i|e_i\lr\tau|^2=-\frac13\sum_i\tr(\tau_{e_i}^2)=-\frac n3a_1\. \ee

By Newton's relations, \eqref{trace} shows that the endomorphisms $(\tau_X)_p$ have the same characteristic polynomial for every $p\in M$ and for every unit vector $X\in \T_p M$. Since they are skew-symmetric, they have the same conjugacy class under the action of the isometry group of $(\T_pM,g_p)$. Thus $\tau_p$ is a generalized vector cross product on $(\T_pM,g_p)$ for every $p\in M$. 

\end{proof}

We are now ready to prove the main result of the paper:

\begin{ath}\label{54} Let $\tau$ be a Killing $3$-form on a compact Riemannian manifold $(M^n,g)$ with negative sectional curvature. 

$(i)$ If $n$ is odd, then either $n=3$ and $\tau$ is parallel, or $n\ne 3$ and $\tau=0$. 

$(ii)$ If $n$ is even, then $\tau=0$.
\end{ath}
\begin{proof} (i) Assume that $\tau$ is not identically 0. Lemma \ref{lt} and Theorem \ref{th} (i) show that after rescaling, $\tau$ defines a vector cross product on $\RM^n$, so by Proposition \ref{vcp}, either $n=3$, or $n=7$.

In the first case, $\dd\tau=0$ by dimensional reasons, so by \eqref{kf} we obtain that $\tau$ is parallel.

In the second case, Proposition \ref{vcp} shows that for every $p\in M$, there exists an isometry $u:\RM^7\to(\T_pM,g_p)$ and a real number $\lambda_p\in\RM$ such that $u^*\tau_p=\lambda_p\tau_0$, where $\tau_0$ is the $3$-form defined in \eqref{tau0}. Moreover, the norm of $\tau$ is constant on $M$ by \eqref{cn}, so $\lambda_p=\lambda$ is independent of $p$. We thus either have $\tau=0$, or $\frac1\lambda\tau$ is the canonical $3$-form of the $\G_2$-structure on $(M,g)$ defined by $\tau_0$. The fact that $\tau$ is a Killing $3$-form is equivalent to the fact that the $\G_2$-structure is nearly parallel
(cf. \cite{uwe}). It is well known that $(M,g)$ is then Einstein with positive scalar curvature (cf. \cite{bfgk}, \cite{fkms}), contradicting the hypothesis that $(M,g)$ has negative sectional curvature.

(ii) Assume that $n$ is even and $\tau$ is not identically $0$. By Lemma \ref{lt}, $\tau$ has constant norm and $\tau_p$ defines a generalized vector cross product on $(T_p M, g_p)$ for every $p\in M$. From Theorem \ref{th} (ii) and (iii) we necessarily have $n=6$ and there exists a non-zero constant $\lambda$ such that $\tau_p$ is conjugate to the $3$-form $\lambda\sigma_0$ defined in \eqref{sigma} for every $p\in M$. Since the stabilizer of $\sigma_0$ in $\O(6)$ is $\SU(3)$, this defines a $\SU(3)$-structure $(\omega,\Psi^\pm)$ on $M$, and after possibly doing a rotation in the space generated by $\Psi^+$ and $\Psi^-$ one can assume that $\tau=\lambda\Psi^-$. 

By Proposition \ref{3k} we either have that $(M,g,\omega)$ is Kähler and $\Psi^\pm$ are parallel, in which case it is well known that $(M,g)$ is Ricci-flat, or $(M,g,\omega)$ is strict nearly Kähler, and has positive scalar curvature by Proposition \ref{pnk}. In both cases we get a contradiction with the assumption that $(M,g)$ has negative sectional curvature, thus finishing the proof.

\end{proof}


\end{document}